\documentclass[12pt]{amsart}
\usepackage{setspace}
\usepackage{moreverb}

\newcommand{\ignore}[1]{}

\newtheorem{theodef}{Theorem/Definition}[section]
\newtheorem{theorem}{Theorem}[section]
\newtheorem{lemma}[theorem]{Lemma}
\newtheorem{corollary}[theorem]{Corollary}
\newtheorem{proposition}[theorem]{Proposition}

\theoremstyle{definition}

\theoremstyle{remark}
\newtheorem{remark}[theorem]{Remark}

\numberwithin{equation}{section}

\input xy
\xyoption{all}

\newcommand{\bN}{{\mathbb{N}}}

\newcommand{\hlt}{\mathrm {ht\, }}

\newcommand{\dm}{\mathrm {dim }}

\newcommand{\depth}{\mathrm {depth\, }}

\newcommand{\length}{\mathrm {length\, }}

\newcommand{\Ext}{\mbox{\rm{Ext}} }

\newcommand{\fM}{{\mathfrak m}}

\newcommand{\cM}{{\mathcal M}}

\newcommand{\la}{{\lambda}}

\newcommand{\lra}{{\longrightarrow}}

\begin{document}

\title[Lyubeznik table of sequentially Cohen-Macaulay rings]{Lyubeznik table of sequentially Cohen-Macaulay rings}

\author[J. \`Alvarez Montaner]{Josep \`Alvarez Montaner$^\dag$}
\thanks{$\dag$Partially supported by SGR2009-1284 and MTM2010-20279-C02-01}
\address{Dept. Matem\`atica Aplicada I\\
Univ. Polit\`ecnica de Catalunya\\ Av. Diagonal 647,
Barcelona 08028, SPAIN} \email{Josep.Alvarez@upc.edu}



\begin{abstract}
We prove that sequentially Cohen-Macaulay rings in positive characteristic, as well as sequentially 
Cohen-Macaulay Stanley-Reisner rings in any characteristic, have trivial Lyubeznik table. Some other configurations 
of Lyubeznik tables are also provided depending on the deficiency modules of the ring.

\end{abstract}

\maketitle

\section{Introduction}

Let $(R,\fM)$  be a regular local ring containing a field $k$ and $I$ an ideal of $R$. Some finiteness properties
of local cohomology modules $H^r_I(R)$ where proved by 
C.~Huneke and R.~Y.~Sharp \cite{HS93}  when the field $k$ has positive characteristic 
and G.~Lyubeznik \cite{Ly93}  in the characteristic zero case (see also \cite{Ly00} for a characteristic-free approach). 
In particular, they proved that Bass numbers of these local cohomology modules are finite.  
Relying on this fact, G.~Lyubeznik \cite{Ly93} introduced a set of numerical invariants of local rings containing a field as follows:

\begin{theodef}
Let $A$ be a local ring containing a field $k$, so that its completion $\widehat{A}$ admits a surjective
ring homomorphism $\xymatrix@1{R\ar@{->>}[r]^-{\pi}& \widehat{A}}$ from a regular local ring
$(R,\mathfrak{m},k)$ of dimension $n$ and set $I:=\ker (\pi)$. Then, the Bass numbers\footnote{The last equality follows from \cite [Lemma 1.4]{Ly93}.}
\[
\lambda_{p,i} (A):=\mu_p (\mathfrak{m},H_I^{n-i}(R))=\mu_0 (\mathfrak{m},H_{\mathfrak{m}}^p (H_I^{n-i}(R)))
\]
depend only on $A$, $i$ and $p$, but neither on $R$ nor on $\pi$.
\end{theodef}

We refer to these invariants as \emph{Lyubeznik numbers} and they are known to
 satisfy the following properties:

\begin{itemize}
\item[(i)]     $\la_{p,i}(A)=0$ if $i>d$.
\item[(ii)]    $\la_{p,i}(A)=0$ if $p>i$.
\item[(iii)]   $\la_{d,d}(A)\neq 0$.
\end{itemize}

\noindent where $d=\dm A$. Therefore we can collect them in
the so-called {\it Lyubeznik table}:
$$\Lambda(A)  = \left(
                    \begin{array}{ccc}
                      \la_{0,0} & \cdots & \la_{0,d}  \\
                       & \ddots & \vdots \\
                       &  & \la_{d,d} \\
                    \end{array}
                  \right)
$$

Despite its algebraic nature, Lyubeznik numbers also provide some geometrical and topological information
as it was already pointed out in \cite{Ly93}. For instance, in the case of isolated singularities, Lyubeznik numbers can 
be described in terms of certain singular cohomology groups in characteristic zero (see \cite{GS98}) or 
\'etale cohomology groups in positive characteristic (see \cite{BlBo}, \cite{Bl3}). 
The highest Lyubeznik number $\la_{d,d}(A)$ can be described using the so-called Hochster-Huneke
graph as it has been proved in \cite{Lyu06}, \cite{Zha07}. However very little is known about 
the possible configurations of Lyubeznik tables except for low dimension cases \cite{Kaw00}, 
\cite{Wa2} or the just mentioned case of isolated singularities.

\vskip 2mm

Local cohomology modules have a natural structure over the ring of $k$-linear differential
operators $D_{R|k}$ (see \cite{Ly93}, \cite{Lyu97}). In fact they are {\it holonomic} $D_{R|k}$-modules
so they have finite length \cite[Thm. 2.7.13]{Bj79}. One may check that Lyubeznik numbers 
are nothing but the length as $D_{R|k}$-module of the local cohomology modules  $H_{\mathfrak{m}}^p (H_I^{n-i}(R))$,
i.e. $$\la_{p,i}(A)= \length_{D_{R|k}}(H_{\mathfrak{m}}^p (H_I^{n-i}(R))).$$ 
From now on we will denote the  $D_{R|k}$-module length simply as $e(-)$.
A key fact that we will use is that the $D_{R|k}$-module length
is an additive function, i.e given a short exact sequence of holonomic $D_{R|k}$-modules $0\lra M_1 \lra M_2 \lra M_3\lra 0$  
we have $$e(M_2)=e(M_1)+e(M_3).$$

The following property of Lyubeznik numbers will play a crucial role in our main result. It was shown to us by R.~Garc\'ia L\'opez in
a graduate course \cite{GL} but we will sketch the proof for the sake of completeness:

\begin{proposition} Lyubeznik numbers satisfy the following  {\it Euler characteristic formula:}
 $$\sum_{0\leq p,i \leq d} (-1)^{p-i} \la_{p,i}(A)=1.$$

\end{proposition}

\begin{proof}
Consider Grothendieck's spectral sequence $$E_2^{p,n-i}= H_{\fM}^p(H_I^{n-i}(R))\Longrightarrow H_{\fM}^{p+n-i}(R).$$
We define the Euler characteristic of the $E_2$-term with respect to the  $D_{R|k}$-module length $e$ as
$$\chi_e(E_2^{\bullet,\bullet})=\sum_{p,i} (-1)^{p+n-i} e(E_2^{p,n-i}).$$ We can also define the Euler characteristic
of the graded $R$-module $H^{\bullet}_{\fM}(R)$ as $$\chi_e(H^{\bullet}_{\fM}(R))=\sum_j (-1)^{j} e(H^{j}_{\fM}(R)).$$
It is a general fact of the theory of spectral sequences that $\chi_e(E_2^{\bullet,\bullet})=\chi_e(H^{\bullet}_{\fM}(R))$
due to the additivity of the $D_{R|k}$-module length.

\vskip 2mm

Therefore, since  $e(E_2^{p,n-i})= e(H_{\fM}^p(H_I^{n-i}(R)))=\la_{p,i}(A)$
and  $e(H_{\fM}^{n}(R))= 1$ we get  $$\chi_e(E_2^{\bullet,\bullet})=\sum_{0\leq p,i \leq d} (-1)^{p+ n-i} \la_{p,i}(A)=(-1)^n=\chi_e(H_{\fM}^{\bullet}(R))$$ and the result follows.
\end{proof}

The first example one may think of Lyubeznik tables is when there is only one local cohomology
module different from zero. Indeed, assume that $H^r_I(R)=0 $ for all $r\neq \hlt I$.
Then, using Grothendieck's spectral sequence  we obtain a
{\it trivial Lyubeznik table}.
$$\Lambda(R/I)  = \left(
                    \begin{array}{ccc}
                      0 & \cdots & 0  \\
                       & \ddots & \vdots \\
                       &  & 1 \\
                    \end{array}
                  \right)
$$
This situation is achieved, among others, in the following cases:
\begin{itemize}
\item[$\cdot$] $R/I$ is Cohen-Macaulay and contains a field of positive characteristic.
\item[$\cdot$] $R/I$ is Cohen-Macaulay and $I$ is a squarefree monomial ideal in any characteristic.
\end{itemize}

\begin{remark}
When $R/I$ is Cohen-Macaulay containing a field of characteristic zero the previous result is no longer true.
For example, consider the ideal generated 
by the $2\times 2$ minors of a generic $2\times 3$ matrix. 
Its Lyubeznik table was computed in \cite{AL06}:
$$\Lambda(R/I)=\begin{pmatrix}
  0 & 0 & 0 & 1 & 0 \\
   & 0 & 0& 0 & 0 \\
   &  & 0 & 0 & 1 \\
   &  &  & 0 & 0 \\
   &  &  &  & 1
\end{pmatrix}$$ We point out that K.~I.~Kawasaki already proved in \cite{Kaw02}  that the highest Lyubeznik number $\lambda_{d,d}$ 
of a Cohen-Macaulay ring (or even $S_2$) is always one.
\end{remark}

\vskip 2mm

The main result of this note is Theorem \ref{main} where we prove that the previous result still holds true replacing the Cohen-Macaulay
property for sequentially Cohen-Macaulay, in particular, assuming that we may have more than one local cohomology
module different from zero. For example, consider the ideal $I=(x_1,x_3)\cap (x_2,x_3)\cap (x_4)$  in $R=k[\![x_1,x_2,x_3,x_4]\!]$. The local cohomology
modules $H^1_I(R)$ and $H^2_I(R)$ are different from zero so $R/I$ is not Cohen-Macaulay but it is sequentially Cohen-Macaulay since
it corresponds to a simplicial tree (see \cite{Far} for details).

\vskip 2mm

In the spirit of \cite{SW}, we give a unified proof of both
cases using the theory of modules over skew-polynomial rings.
We point out that the case of squarefree monomial ideals is already treated in a joint work with
K.~Yanagawa \cite{AY13} using the description of Lyubeznik numbers of squarefree monomial ideals in
terms of the linear strands of the Alexander dual ideal given in \cite{AV11}. Finally, in the last section, we use 
the same techniques to provide some configurations of Lyubeznik tables depending on the 
so-called {\it deficiency modules} $$K^i(R/I):=\Ext^{n-i}_R(R/I,R).$$

\vskip 2mm 

Sequentially Cohen-Macaulay modules were introduced by R.~Stanley \cite{Sta96} in the graded case but
extended later on to the local case. We present here a homological characterization, due to C.~Peskine (see \cite{HS02})
in the graded case and  P.~Schenzel \cite{Sch99} in the local case
(see also \cite{CN03}), that will be useful for our purposes.
We decided to consider just the case of
regular local rings to keep the same framework as in the rest of the paper.

\begin{theodef}
 Let $(R,\fM)$ be a regular local ring of dimension $n$. Then, an $R$-module $M$ is sequentially Cohen-Macaulay
if and only if for all $0\leq i \leq \dim M$ we have that $\Ext^{n-i}_R(M,R)$ is zero or Cohen-Macaulay of dimension $i$.

\end{theodef}

In our situation we will be interested in the case when the $R$-module $M$ is just the local ring 
$ R/I$ for any given ideal $I\subseteq R$.
We also point out that throughout this work we will freely use some standard facts about local cohomology modules.
We refer to \cite{BS98} for any unexplained terminology.

\section{Finitely generated unit $R[\Theta; \varphi]$-modules}
Let $(R,\fM)$ be a regular local ring of dimension $n$ containing a
field $k$. Throughout the rest of the paper we will assume that we have a flat local endomorphism
$\varphi:R \lra R$ satisfying, for a given ideal $I\subseteq R$, the condition:

\vskip 2mm

$(\ast)$ The ideals $\{\varphi^t(I)R\}_{t\geq 0}$ form a descending
chain cofinal with the chain $\{I^t\}_{t\geq 0}$.

\vskip 2mm

\noindent Notice that in this case, by the dimension
formula, we have that $\varphi^t(\fM)R$ is $\fM$-primary. The main examples we are going
to consider are:

\vskip 2mm

\begin{itemize}
\item[$\cdot$] {\it Positive characteristic case:} When $R$ contains a field of positive characteristic,
the Frobenius endomorphism $\varphi=F$ satisfies $(\ast)$ for any ideal $I\subseteq R$ (see  \cite{HS93}, \cite{Lyu97}) 
and is flat by the celebrated theorem of E.~Kunz \cite{Kun69}.

\vskip 2mm

\item[$\cdot$] {\it Squarefree monomial ideals case:} Consider the polynomial ring $R=k[x_1,\dots,x_n]$ over a field $k$
in any characteristic. The $k$-linear endomorphism $\varphi(x_i)=x_i^2$ is flat and satisfies $(\ast)$ for any
squarefree monomial ideal (see \cite{Ly84}).
\end{itemize}

\vskip 2mm

G.~Lyubeznik \cite{Lyu97} developed his theory of $F$-modules in positive characteristic building upon these properties 
for the Frobenius map. One may give a slightly extended theory associated to the morphism $\varphi:R\lra R$ instead of the 
Frobenius as follows:

\vskip 2mm

 Let $\varPhi$ be the functor on the category of $R$-modules defined by restriction of scalars. Namely, for any 
 $R$-module $M$, $\varPhi(M)$ is the additive group of $M$ with the usual action of $R$ on the right but regarded 
 as a left $R$-module via $\varphi$.
 Notice that we can also construct the $e$-th iterations $\varPhi^e$ in the usual way.


\vskip 2mm

Let $R[\Theta; \varphi]$ be the skew polynomial ring which is the free left $R$-module $\bigoplus_{e\geq 0} R\Theta^e$
with multiplication $\Theta r= \varphi(r) \Theta$. In fact we have
$$R[\Theta; \varphi]=R\langle \Theta \rangle/\langle \Theta r - \varphi(r) \Theta \hskip 2mm | \hskip 2mm r\in R\rangle.$$

To give a $R[\Theta; \varphi]$-module structure on a $R$-module $\cM$ is equivalent to fix a $R$-linear map
$\theta_{\cM}: \cM \lra \Phi(\cM)$. We say that $\cM$ is a {\it unit $R[\Theta; \varphi]$-module} if $\theta_{\cM}$
is an isomorphism.

\vskip 2mm

Given a finitely generated $R$-module $M$ and a $R$-linear map $\beta: M\lra \varPhi(M)$ one can obtain a
unit $R[\Theta; \varphi]$-module
$$\cM:= {\rm Gen}(M)= \varinjlim (\xymatrix{M \ar[r]^{\beta} & \varPhi(M) \ar[r]^{\varPhi(\beta)}  &
\varPhi^2(M) \ar[r]^{\varPhi^2(\beta)} & \cdots} )$$ just because
$$\varPhi(\cM):= {\rm Gen}(\varPhi(M))= \varinjlim (\xymatrix{\varPhi(M) \ar[r]^{\varPhi(\beta)} &
\varPhi^2(M) \ar[r]^{\varPhi^2(\beta)}  &
\varPhi^3(M) \ar[r]^{\varPhi^3(\beta)} & \cdots} ) = \cM$$
We say that $\cM$ is a  {\it finitely generated unit $R[\Theta; \varphi]$-module} if it can constructed
in this way\footnote{A ($F$-finite) $F$-module
in the sense of G.~Lyubeznik \cite{Lyu97} is a (finitely generated) unit $R[\Theta; F]$-module.}. Moreover, if
the generating morphism $\beta$ is injective, we say that $M$ is a {\it root} of $\cM$. The main example
we are going to consider, that was already treated by A.~Singh and U.~Walther in \cite{SW}, is the case of local cohomology modules.

\vskip 2mm

As it was already stated in \cite{SW}, the flatness of the morphism $\varphi$ implies that $\varPhi$ is an exact functor
and commutes with direct limits. It also follows that $\varPhi^e(\Ext^{i}_R(R/I,R))\cong \Ext^{i}_R(R/\varphi^e(I)R,R)$. We also have a commutative diagram
$$\xymatrix{\dots\ar[r]& \Ext_R^i (R/\varphi^e (I)R,R)\ar[d]\ar[r]& \Ext_R^i (R/\varphi^{e+1} (I)R,R)\ar[d]\ar[r]& \dots\\
\dots\ar[r]& \Phi^e (\Ext_R^i (R/I,R))\ar[r]& \Phi^{e+1} (\Ext_R^i (R/I,R))\ar[r]& \dots}$$
where the maps in the top row are induced by the natural surjection
$$R/\varphi^{e+1}(I)R \lra R/\varphi^{e}(I)R$$ and the vertical maps are isomorphisms. Taking into account property
$(\ast)$, the limit of the top row is the local cohomology module $H_I^i(R)$. We conclude that
local cohomology modules are finitely generated unit  $R[\Theta; F]$-modules and the generating morphism
$$\beta: \Ext^{i}_R(R/I,R) \lra \Ext^{i}_R(R/\varphi(I)R,R)$$ is induced by the natural surjection
$R/\varphi(I)R \lra R/I$.

\begin{remark}
Under this terminology, \cite[Thm. 2.8]{SW} states that $\Ext^{i}_R(R/I,R)$ is a root of $H_I^i(R)$ when
 the induced morphism $\overline{\varphi}:R/I \lra R/I$ is pure.

\end{remark}

\section{Main result}

The description of local cohomology modules given in Section $2$ allows us to obtain
the main result of this note, but first we consider the following vanishing result
for Bass numbers that is a mild generalization of \cite[Thm. 3.3]{HS93}.

\begin{lemma} \label{HS}
 Let $(R,\fM)$ be a regular local ring of dimension $n$ containing a field $k$ and $\varphi:R \lra R$ a flat local endomorphism
 satisfying  $(\ast)$ for an ideal $I \subseteq R$. Given $p,i\in \bN$, if $H_\fM^p(\Ext^{n-i}_R(R/I,R))=0$ then 
$\mu_p(\fM, H^{n-i}_I(R))=0$.
\end{lemma}

\begin{proof}
 Using flat base change for local cohomology and the fact that $\varphi^t(\fM)R$ is $\fM$-primary we have:
$$\varPhi^e(H_\fM^p(\Ext^{n-i}_R(R/I,R)))\cong H_\fM^p(\varPhi^e(\Ext^{n-i}_R(R/I,R)))\cong H_\fM^p(\Ext^{n-i}_R(R/\varphi^e(I)R,R)).$$
Therefore, since local cohomology commutes with direct limits
$$H_\fM^p(H^{n-i}_I(R)) \cong H_\fM^p(\varinjlim \varPhi^e(\Ext^{n-i}_R(R/I,R)))\cong
\varinjlim  \varPhi^e(H_\fM^p(\Ext^{n-i}_R(R/I,R)))$$ we get the desired result.
\end{proof}

\begin{theorem} \label{main}
Let $(R,\fM)$ be a regular local ring of dimension $n$ containing a field $k$ and $\varphi:R \lra R$ a flat local endomorphism
 satisfying  $(\ast)$ for an ideal $I \subseteq R$ such that 
$R/I$ is sequentially Cohen-Macaulay. Then the Lyubeznik table of $R/I$ is trivial.
\end{theorem}

\begin{proof}
If $R/I$ is sequentially Cohen-Macaulay then we have that $\Ext^{n-i}_R(R/I,R)$ is zero or Cohen-Macaulay of dimension $i$.
Therefore $H_\fM^p(\Ext^{n-i}_R(R/I,R))=0$ for all $p\neq i$. It follows from Lemma \ref{HS} that the possible non-zero Lyubeznik numbers are
$\lambda_{i,i}(R/I)$, i.e. those in the main diagonal of the Lyubeznik table. Using the Euler characteristic formula
for Lyubeznik numbers and property (iii) we have $\lambda_{0,0} +\cdots + \lambda_{d,d}=1$ and $\lambda_{d,d}\neq 0$ so we must have a trivial Lyubeznik table.
\end{proof}

\begin{remark}
The completion with respect to the maximal ideal of a sequentially Cohen-Macaulay ring is sequentially Cohen-Macaulay \cite[Thm. 4.9]{Sch99} but
the converse does not hold as P. Schenzel showed in \cite[Ex. 6.1]{Sch99} using Nagata's example \cite[Ex.2]{Na62}.
Lyubeznik numbers does not depend on the completion so we can just assume that the completion of $R/I$ is sequentially
Cohen-Macaulay in the hypothesis of Theorem \ref{main}.
\end{remark}

Specializing to the cases considered at the beginning of Section $2$ we obtain:

\begin{corollary}
Let $(R,\fM)$ be a regular local ring containing a field $k$. Then the Lyubeznik table of $R/I$ is trivial
in the following cases:
\begin{itemize}
\item[$\cdot$] $R/I$ is sequentially Cohen-Macaulay and contains a field of positive characteristic.
\item[$\cdot$] $R/I$ is sequentially Cohen-Macaulay and $I$ is a squarefree monomial ideal.
\end{itemize}

\end{corollary}

\begin{remark}

As it was already pointed out in \cite{AV11}, the converse statement does not hold. For example consider
the ideal in $k[\![x_1,\dots,x_9]\!]$:

\vskip 2mm

$I=(x_1,x_2)\cap(x_3,x_4)\cap (x_5,x_6) \cap (x_7,x_8)\cap (x_9,x_1)\cap (x_9,x_2)\cap (x_9,x_3)\cap (x_9,x_4)\cap (x_9,x_5)\cap $

$\hskip .8cm  \cap (x_9,x_6)\cap (x_9,x_7)\cap (x_9,x_8).$

\vskip 2mm

\noindent $R/I$  has a trivial Lyubeznik table but is not sequentially Cohen-Macaulay.
We remark that  $H^r_I(R)$ does not vanish for
$r=2,3,4,5$.

\end{remark}

\section{Some partial vanishing results}
A way to measure the deviation of $R/I$ from being Cohen-Macaulay is through 
the {\it deficiency modules} $$K^i(R/I):=\Ext^{n-i}_R(R/I,R).$$ Notice that for $d=\dim R/I$ we have that $K^d(R/I)$ 
is nothing but the canonical module. In this sense, sequentially Cohen-Macaulay rings 
form a class where these deficiency modules are well understood. The methods developed in the previous section suggest that 
some configurations of Lyubeznik tables could be described depending on the behavior of these modules.

\vskip 2mm

In this direction we recall the following notion developed by P.~Schenzel in \cite{Sch04}: We say that 
$R/I$ is {\it canonically Cohen-Macaulay}  (CCM for short) if the canonical module 
$K^d(R/I)$ is Cohen-Macaulay.

\begin{proposition}
Let $(R,\fM)$ be a regular local ring of dimension $n$ containing a field $k$ and $\varphi:R \lra R$ a flat local endomorphism
 satisfying  $(\ast)$ for an ideal $I \subseteq R$ such that
$R/I$ is canonically Cohen-Macaulay. Then, $\lambda_{i,d}(R/I)=0$ for all $i< d$.
\end{proposition}

\begin{proof}
The canonical module $K^d(R/I)=\Ext^{n-d}_R(R/I,R)$  is Cohen-Macaulay of dimension $d$  by \cite[Prop. 2.3]{Sch04}, so the result 
follows from Lemma \ref{HS}.
\end{proof}

For a general description of the highest Lyubeznik number we refer to \cite{Lyu06}, \cite{Zha07} 
where $\lambda_{d,d}(R/I)$ is described as the number of connected components of the Hochster-Huneke graph of 
the completion of the strict Henselianization of $R/I$. 

\vskip 2mm

Examples of CCM modules include Cohen-Macaulay and sequentially Cohen-Macaulay modules among others (see \cite[Ex.3.2]{Sch04}).
Using Theorem \ref{main} we have that $\lambda_{d,d}(R/I)=1$ in these cases but, of course, we may find examples
of CCM rings where this number is larger. For instance,
the ideal $I=(x_1,x_2)\cap(x_3,x_4)$ in $k[\![x_1,x_2,x_3,x_4]\!]$ satisfies that $R/I$ is CCM and its Lyubeznik table is
$$\Lambda(R/I)  = \left(
                    \begin{array}{ccc}
                      0 & 1 & 0  \\
                       & 0 & 0 \\
                       &  & 2 \\
                    \end{array}
                  \right)
$$
This example can be seen as a particular case of the following result.

\begin{proposition}
Let $(R,\fM)$ be a regular local ring of dimension $n$ containing a field $k$ and $\varphi:R \lra R$ a flat local endomorphism
 satisfying  $(\ast)$ for an ideal $I \subseteq R$ such that 
$R/I$ is unmixed and $\depth K^i(R/I)\geq i-1$ for $0\leq i < d$. Then, its Lyubeznik table is of the form
$$\Lambda(R/I)  = \left(
                    \begin{array}{ccccc}
                      0 & \la_{0,1} &\cdots & 0 & 0  \\
                       & 0  &\cdots & 0 & 0  \\
                        &   &\ddots & \vdots & \vdots  \\
                       & & & \la_{d-2,d-1}& 0 \\
                       & & & 0 & 0 \\
                       & & & & \la_{d,d} \\
                    \end{array}
                  \right)
$$
where $ \la_{0,1} + \cdots + \la_{d-2,d-1}= \la_{d,d}-1$. In particular, the Lyubeznik table is trivial when the highest 
Lyubeznik number is $1$.

\end{proposition}

\begin{proof}
By \cite[Thm.4.4]{Sch04} we have that $R/I$ is CCM and  $ K^i(R/I)$ is either zero or Cohen-Macaulay of dimension $i-1$.
Then, using Lemma \ref{HS} we get the desired Lyubeznik table. The rest of the proof follows from the Euler characteristic 
property  of Lyubeznik numbers.
\end{proof}

Another large class of CCM rings discussed in \cite{Sch04} is the case of simplicial affine semigroup rings (see \cite[Thm. 6.4]{Sch04}). 
Let $S$ be a  finitely generated submonoid of $\bN^r$. The affine semigroup $k[S]$ of $S$ over $k$ is the subring of 
$k[x_1,\dots,x_r]$  generated
by all monomials $x^s:=x_1^{s_1}\cdots x_r^{s_r}$, $s\in S$. Equivalently, if $n$ is the minimal number of generators of $S$, we 
may write $k[S]=R/I(S)$ where $R=k[x_1,\dots,x_n]$  and $I(S)$  is the ideal of vanishing of $k[S]$. 
We say that $S$ is simplicial if there is a homogeneous
system of parameters of $k[S]$ with $d=\dim k[S]$ elements.

\begin{proposition}
Let $k[S]$ be a simplicial affine semigroup ring of codimension $2$, i.e. $n=d+2$ and 
$\varphi:R \lra R$ a flat endomorphism satisfying  $(\ast)$ for the ideal of vanishing $I(S)\subseteq R$ of $k[S]$.
If the number of generators of this ideal is $m\leq 3$, its Lyubeznik table is trivial. Otherwise it is of the form
$$\Lambda(k[S])  = \left(
                    \begin{array}{cccc}
                      0 & \cdots & 0 & 0  \\
                    
                           &\ddots & \vdots & \vdots  \\
                       & & \la_{d-2,d-1}& 0 \\
                        & & 0 & 0 \\
                        & & & \la_{d,d} \\
                    \end{array}
                  \right)$$
where $\la_{d-2,d-1}= \la_{d,d}-1$. 

\end{proposition}

\begin{proof}
By \cite[Thm.6.5]{Sch04} we have that $k[S]$ is Cohen-Macaulay if and only if $m\leq 3$. When $m>3$ we have that
 $ K^i(k[S])=0$  for all $0\leq i<d-1$ and $ K^{d-1}(k[S])$ is either zero or Cohen-Macaulay of dimension $d-2$.
Then, using Lemma \ref{HS} we get the desired Lyubeznik table. 
\end{proof}

{\bf Aknowledgement: } The author benefited from conversations with K.~Yanagawa. He also thanks the anonymous referee
for a careful reading and useful comments.

\end{document}